\documentclass[12pt]{article}
\usepackage{amsfonts}
\usepackage{amsmath,amsthm}
\usepackage{cases}

\newtheorem{lemma}{Lemma}[section]

\newtheorem{theorem}{Theorem}[section]
\newtheorem{proposition}[theorem]{Proposition}
\theoremstyle{definition}
\newtheorem{definition}{Definition}[section]

\newtheorem{remark}{Remark}

\newcommand{\comment}[1]{}

\numberwithin{equation}{section}
\begin{document}
\title{Buser's inequality on infinite graphs}

\author{Shuang Liu}
\date{}
\maketitle


\begin{center}
\textbf{Abstract}
\end{center}
In this paper, we establish Buser type inequalities, i.e., upper bounds for eigenvalues in
terms of Cheeger constants.
We prove the Buser's inequality for an infinite but locally finite connected graph with Ricci curvature lower bounds. Furthermore, we derive that the graph with positive curvature is finite, especially for unbounded Laplacians. By proving Poincar\'e inequality, we obtain a lower bound on Cheeger constant in terms of positive curvature.

\textbf{Keywords:}~~ Buser's inequality, infinite graph, curvature-dimension inequality, Cheeger constants.
\section{Introduction}

In 1982, Peter Buser \cite{B82} proved an upper bound for the first nonzero eigenvalue $\lambda _{1}$ of the Laplace-Beltrami operator in terms of the Cheeger constant $h$. Let $M$ be an $n$-dimensional compact Riemannian manifold whose Ricci curvature is bounded from below by $-(n-1)a^2$, where $a \geq 0$. Then
\[ \lambda _{1}\leq c(ah+h^{2}),\]
where $c$ is a constant depending on the dimension.

The authors of \cite{KKRT15} extend an argument of Ledoux [27] to establish analogous Buser
type estimates on finite graphs satisfying the classical curvature dimension inequality $CD(\infty,0)$. This result was extended to bound the higher eigenvalues of the Laplacian, see \cite{LMP15}. By proving a discrete version of Li-Yau inequality, F.Bauer et.al. \cite{BHLLMY15} also obtained Buser's inequality on a finite graph under exponential curvature dimension conditions $CDE(n,0)$.

For infinite graphs, lots of researches have been focused on Cheeger type inequality, which states a lower bound for the bottom of the $\ell^2$-spectrum of the graph Laplacian 
in terms of Cheeger constant on the whole graph, e.g. see \cite{BKW12, KL10}.
However, no attempt has been made here to develop Buser's inequality on infinite graphs. In this paper, we study all weighted graphs, including bounded and unbounded Laplacians, and obtain Buser's inequality on infinite graphs under the assumption of curvature-dimension inequality.

For positively curved graphs, P. Horn et.al. \cite{BHLLMY15} proved Bonnet-Myers type theorem,
which states that the diameter of graphs in terms of the canonical distance is finite. By a more direct proof, the authors of \cite{LMP16} obtained that the diameter is finite, for operators including bounded and undounded laplacians. It is obvious that the graph equipped with bounded Laplacians is finite. However, it is insufficient to prove the finiteness when the Laplacian is unbounded. In this paper, we are led to the strengthening of the results in \cite{LMP16} under the same hypotheses. As a consequence, we derive Poincar\'e inequality, and then a lower bound for Cheeger constant.

The paper is organized as follows: We introduce our notations and formally state our main
results in Section 2. In Section 3, we prove the equivalence of the semigroup properties and the curvature-dimension inequality. As a consequence, we derive Pseudo-Poincar\'e inequality, and then, combining with the upper bound of the heat kernel in \cite{BHY17}, then derive Buser's inequality on infinite graphs. Next, we obtain that positively curved graphs are finite, and give a lower bound on Cheeger constant in terms of curvature.

\textbf{Acknowledgement} The author is supported by the Certificate of China Postdoctoral   Science foundation Grant (Grant No.2018M631435). The author gratefully acknowledges the many helpful suggestions of Professor Yong Lin during the preparation of the paper.

\section{Notation and main results}
The basic setting is as follows. Let $G=(V,E,m,\omega)$ be an connected weighted graph, we denote by $V$ the set of vertices and by $E$ the set of edges, which is a symmetric subsets of $V\times V$. Two vertices are called neighbours if they are connected by an edge $(x,y)\in E$, which is denoted by $x\sim y$. Let $m:V\rightarrow\mathbf{R^+}$ be the vertex measure and $\omega:V\times V\rightarrow \mathbf{R_{\geq0}}$ be the edge weights, so that $\omega_{xy}>0$ if $x$ and $y$ are neighbours, otherwise $\omega_{xy}$ equals to $0$. We require that the edge weights be symmetric, namely, $\omega_{xy}=\omega_{yx}$. We restrict our interest in locally finiteness, i.e. the number of neighbours of any vertex is finite.

Let $V^\mathbf{R}$ be the set of functions $f:V\rightarrow\mathbf{R}$, and
$C_0(V)$ be the set of finitely supported functions on $V$. We denote by
$\ell^p_m, p\in [1,\infty)$, the spaces of integrable functions on $V$ with respect to the measure $m$, and by $\ell^\infty$ the set of bounded functions. We further denote by $\|\cdot\|_{\ell^p_m}$ the $p$-norm of a function, and by $\|\cdot\|_\infty$ the $\ell^\infty$-norm of a function in standard ways. For any $f,g\in \ell^2_m$, we let $\langle f,g\rangle=\sum_{x\in V}f(x)g(x)m(x)$ denote the standard inner product. This makes $\ell^2_m$ a Hilbert space.

A weighted graph $G$ is associated with a Dirichlet form with respect to $\ell_m^2$,
\[\begin{split}
Q:D(Q)\times D(Q)&\rightarrow \mathbb{R}\\
(f,g)&\mapsto Q(f,g):=\frac{1}{2}\sum_{x,y\in V}\omega_{xy}(f(y)-f(x))(g(y)-g(x)),
\end{split}\]
where $D(Q)$ is defined as the completion of $C_0(V)$ under
the $Q$-norm $\|\cdot\|_Q$ given by
\[\|f\|_Q=\sqrt{\|f\|_{\ell^2_m}^2+\frac{1}{2}\sum_{x,y\in V}\omega_{xy}(f(y)-f(x))^2},\]
see \cite{KL12}. For locally finite graphs, the associated generator $\Delta$, called Laplacian, is defined by
\[\Delta f(x)=\frac{1}{m(x)}\sum_{y\sim x}\omega_{xy}(f(y)-f(x)), ~~~~f\in D(\Delta),\]
where $D(\Delta)=\{f\in \ell_m^2|\Delta f\in \ell_m^2\}$. The Laplacian $\Delta$ generates a semigroup $P_t :=e^{t\Delta}$ acting on $\ell^p_m$ for $p\in[1,\infty]$.  Obviously, the measure $m$ plays an important role in the definition of the Laplacian. Given the weight $\omega$ on $E$, there are two typical choices of Laplacian:
\begin{itemize}
  \item $m(x)=\deg(x):=\sum_{y\sim x}\omega_{xy}$ for all $x\in V$, which is called the normalized  Laplacian;
  \item $m(x)\equiv 1$ for all $x\in V$, which is the combinatorial Laplacian.
\end{itemize}

In order to introduce curvature dimension inequality on graphs, we recall the gradient form $\Gamma:V^\mathbf{R}\times V^\mathbf{R}\rightarrow V^\mathbf{R}$ associated to the Laplacian $\Delta$. For any $x\in V$, and functions $f,g\in V^\mathbf{R}$,
\[\begin{split}
2\Gamma(f,g)(x)
&=\Delta(fg)(x)-f(x)(\Delta g)(x)-g(x)(\Delta f)(x)\\
&=\frac{1}{m(x)}\sum_{y\sim x}\omega_{xy}(f(y)-f(x))(g(y)-g(x)).
\end{split}\]
The iterated gradient form $\Gamma_2$ is defined by iterating $\Gamma$ as
\[2\Gamma_2(f,g)(x)=\Delta\Gamma(fg)(x)-\Gamma(f,\Delta g)(x)-\Gamma(g,\Delta f)(x).\]
For simplification, we write $\Gamma(f)=\Gamma(f,f), \Gamma_2(f)=\Gamma_2(f,f)$.

\begin{definition}[Curvature-dimension inequality]
The graph $G$ satisfies $CD(n,K)$ if, for any function $f\in V^\mathbb{R}$ and at every vertex $x \in V$
\begin{equation*}
\Gamma_{2}(f)(x)\geq \frac{1}{n}(\Delta f)^{2}(x)+K\Gamma(f)(x). \label{eqn:cd}
\end{equation*}
\end{definition}

In this paper, we say that a weighted graph G satisfies the assumption $(A)$ if one of the following holds:\\
$(A_1)$ The Laplacian $\Delta$ is \textit{bounded} on $\ell^2_m$, which is equivalent to
$$\sup_{x\in V}\frac{\deg(x)}{m(x)}<\infty,$$
see \cite{KL12}.\\
$(A_2)$ $G$ is \textit{complete}, that is, there exists a non-decreasing sequence $\{\eta_k\}_{k=0}^{\infty}\in C_0(V)$ such that
\begin{equation*}\label{eq:complete}
\lim\limits_{k\rightarrow\infty}\eta_k=\mathbf{1}, ~~ \mbox{and}~~\Gamma(\eta_k)\leq\frac{1}{k},
\end{equation*}
where $\mathbf{1}$ is the constant function on $V$, and the limit is pointwise. Moreover, $m$ should be \textit{non-degenerate}, i.e.
$\inf\limits_{x\in V}m(x)=\delta>0.$

Note that the normalized Laplacian is bounded, while the combinatorial graph Laplacian may not be bounded. In this paper, we consider both bounded and unbounded Laplacians in general on infinite graphs. In order to deal with unbounded Laplacians, it is often useful to introduce the following intrinsics metrics, see \cite{FLW14}.
\begin{definition}[Intrinsic metric]
A pseudo metric $\rho:V\times V\rightarrow [0,\infty)$ is a symmetric function with zero diagonal satisfying the triangle inequality. A pseudo metric $\rho$ is called intrinsic if
\[\sum_{y\in V}\omega_{xy}\rho^2(x,y)\leq m(x), \forall x\in V.\]
The jump size s of a pseudo metric $\rho$ is given by
\[s:=\sup\{\rho(x,y):x,y\in V, x\sim y \}\in [0,\infty].\]
\end{definition}
It is easy to check that the combinatorial distance $d$ is an intrinsic metric for the normalized Laplacian. For a weighted graph, one can always construct an intrinsic metric on it, see e.g. \cite{H11}.

Next, we introduce Cheeger constant and the bottom of the $\ell^2$-spectrum of the graph Laplacian $\Delta$, see \cite{KL10, W09}. Given a subset $\Omega\subset V$, we define the volume $|\Omega|=\sum_{x\in \Omega}m(x)$. And we denote by $\partial \Omega$ the edge boundary set of  $\Omega$, that is
\[\partial \Omega=\{(x,y)\in \Omega\times V\backslash \Omega: \omega_{xy}>0\}.\]
Since the graph $G$ is connected, $\partial \Omega\neq\emptyset$. We define the volume of the boundary $\partial \Omega$ by
\[|\partial \Omega|=\sum_{x\in \Omega,y\in V\backslash\Omega}\omega_{xy}.\]

\begin{definition}[Cheeger constant for finite graphs]
Cheeger constant is defined by
\[h=\min_{U\subset V}\frac{|\partial U|}{\min\{|U|,|V/U|\}}.\]
\end{definition}
\begin{definition}[Cheeger constant for infinite graphs]
For any subset $\Omega\subset V$, define its Cheeger constant by
\[h(\Omega)=\inf_{\emptyset\neq U\subset \Omega ~\mbox{finite}}\frac{|\partial U|}{|U|}.\]
\end{definition}
\begin{definition}[The bottom of the $\ell^2$-spectrum of the Laplacian ]
The bottom of the $\ell^2$-spectrum of the Laplacian $\Delta$ is  defined by
\[\lambda=\inf_{f\in D(Q)\backslash \{0\}}\frac{\langle f,-\Delta f\rangle}{\langle f,f \rangle}=\inf_{f\in C_0(V)\backslash \{0\}}\frac{\langle f,-\Delta f\rangle}{\langle f,f \rangle}.\]
\end{definition}
We further assume that
\[D_\omega:=\sup_{x\in V}\sup_{y,y\sim x}\frac{m(x)}{\omega_{xy}}<\infty.\]
\begin{theorem}\label{th:buser}
If the infinite graph $G=(V,E,\mu,\omega)$ satisfies (A) and curvature-dimension inequality $CD(\infty,K)$ with some $K\in \mathbb{R}$,
then we have
\[ \lambda\leq\max\left\{C_1 h^2(V),C_2\sqrt{|K|}h(V)\right\},\]
for constants $C_1,C_2$ only depending on $D_\omega$.
\end{theorem}
\begin{theorem}\label{thm:finite}
Suppose that the weighted graph $G=(V,E,\mu,\omega)$ satisfies (A) and curvature-dimension inequality $CD(n,K)$ with some $K>0$, then the measure of the graph is finite.

Furthermore, for any $x\in V$ and $f\in C_0(V)$,
$P_tf(x)$ converges to a constant  when $t\rightarrow+\infty$. More precisely,
\begin{equation}\label{eq:infinity}
P_tf(x)\rightarrow\frac{1}{|V|}\sum_{x\in V}m(x)f(x),~~t\rightarrow+\infty.
\end{equation}
\end{theorem}
\begin{remark}Corollary 2.2 and Theorem 2.4 in \cite{LMP16} show that the diameter is bounded if the graph satisfies $CD(n,K)$ with $K>0$. It is obvious that the graph is finite for bounded Laplacians.  For unbounded Laplacians, we have
\[\delta\cdot \sharp V<\sum_{x\in V}m(x)=|V|<+\infty,\]
which means the graph is finite.

\end{remark}
\begin{theorem}\label{thm:chegger}
Under the above assumptions, we have
\[h\geq\frac{1}{2\pi \sqrt{2D_\omega}}\sqrt{K}.\]
\end{theorem}

\section{The proof of the main theorem}
The following integration by parts formula will be useful later, see  \cite{KL10}.
\begin{lemma}[Green's formula]\label{lem:green}
For any $f\in D(Q)$ and $g\in D(\Delta)$ we have
$$\sum\limits_{x\in V}f(x)\Delta g(x)m(x)=-\sum\limits_{x\in V}\Gamma(f,g)(x)m(x).$$
\end{lemma}
The next proposition is a consequence of standard Dirichlet form theory,
see \cite{FOT11} and \cite{KL12}.

\begin{proposition}\label{pro:semi1}
For any $f\in \ell^p_m,p\in[1,\infty]$, we have $P_tf\in\ell^p_m$ and
\[\|P_tf\|_{\ell^p_m}\leq\|f\|_{\ell^p_m},~~~~\forall t\geq0.\]
Moreover, $P_tf\in D(\Delta)$ for any $f\in \ell^2_m$.
\end{proposition}

Now, we introduce some useful properties of semigroup $P_t$ on graphs.

\begin{proposition} \label{pro:semi}
For any $t,s>0,$ we have
\begin{itemize}
  \item $P_t$ and $\Delta$ are commuting, i.e. for any $f\in D(\Delta)$,
\[\Delta P_tf=P_t\Delta f.\]
  \item $P_t$ satisfies the semigroup property. That is, for any $f\in \ell^p_m$,
\[(P_{t}\circ P_s)f=P_{t+s}f.\]
  \item $P_t$ is self-adjoint. That is,  for any $f,g \in \ell^2_m$,
\[\langle P_t f,g\rangle=\langle P_t g,f\rangle.\]
\end{itemize}
\end{proposition}
By proving the Caccioppoli inequality for subsolutions to Poisson＊s equations, the authors derived a uniform upper bound for $P_tf$, see \cite{HL17}.
\begin{lemma}\label{lem:0}
Let $G=(V,E,m,\omega)$ be a complete graph, and $m$ be a non-degenerate measure. For any $f\in C_0(V)$ and $T>0$, we have $\max\limits_{[0,T]}\Gamma(P_tf)\in \ell_m^1$.
\end{lemma}
The following functional formulation of DGG Lemma on graphs is the key for proving our main theorems in this paper, see Theorem 1.1 in \cite{BHY17}.
\begin{lemma}\label{lem:heatkernel}
Let G be a weighted graph with an intrinsic metric $\rho$ with finite jump size
$s > 0$. Let $A,B$ be two subsets in $V$ and $f, g\in \ell^2_m$ with supp $f\subset A$, supp $g\subset B$, then
\begin{equation*}\label{eq:heatkernel}
\langle e^{t\Delta} f,g\rangle\leq e^{-\lambda t-\zeta_s(t,\rho(A,B))}\|f\|_{\ell^2_m}\|g\|_{\ell^2_m},
\end{equation*}
where
\[\zeta_s(t,r)=\frac{1}{s^2}\left(rs\cdot arcsinh\frac{rs}{t}-\sqrt{t^2+r^2s^2}+t\right), t>0,r\geq 0,\]
and $\rho(A,B) = \inf_{x\in A,y\in B} \rho(x, y)$ the distance between $A$ and $B$.
\end{lemma}
Now we introduce some equivalent properties of $CD$ inequality, including gradient bounds and reverse Poincar\'e inequality.
\begin{proposition}\label{pro:eqi}
For an infinite graph $G=(V,E,m,\omega)$ satisfying (A),
the following properties are equivalent:
\begin{enumerate}
  \item[(1)] $CD(n,K)$ for some $K\in \mathbf{R}$.
  \item[(2)] $\Gamma(P_tf)\leq e^{-2Kt}P_t(\Gamma(f))-\frac{1-e^{-2Kt}}{Kn}(\Delta P_tf)^2$, for any  $f\in C_0(V)$.
  \item[(3)] $P_t(f^2)-(P_t f)^2\geq2\int_0^te^{2Ks}ds\Gamma(P_t f)+2\int_0^t\frac{e^{2Ks}-1}{Kn}ds(\Delta P_tf)^2$, for any  $f\in C_0(V)$.
\end{enumerate}
\end{proposition}
\begin{proof}
For bounded Laplacians, the theorem has been proved
in Theorem 3.1 of [LL15]. It suffices to prove the result for unbounded Laplacians under the assumption of $(A_2)$. For unbounded Laplacians, the equivalence of (1) and (2) has been obtained in Theorem 1.1 of [HL17]. In term of (1) and (3), the case when the dimension is infinite was proved in Theorem 2 of [H17].

Now, we prove (1) $\Rightarrow$ (3). The reverse direction is omitted here since it follows verbatim. We consider the following function on $[0,t]\times V$,
\[\phi(s,x)=P_s(P_{t-s}f)^2(x), f\in C_0(V).\]
Since $\Delta$ and $P_t$ are commuting, and by direct calculation, one has
\[\frac{d}{ds} \phi=\Delta P_s(P_{t-s})^2-P_s(2P_{t-s}\Delta P_{t-s})=P_s[\Delta(P_{t-s})^2-2P_{t-s}\Delta P_{t-s}]=2P_s(\Gamma(P_{t-s}f)).\]
To end this, we further need to prove $(P_{t-s}f)^2\in D(\Delta).$ On one hand, from Proposition \ref{pro:semi1}, one can see that $(P_{t-s}f)^2\in\ell^2_m$ since $f\in C_0(V)\subset\ell^1_m$. On the other hand,
\[\Delta (P_{t-s}f)^2=2P_{t-s}f\Delta P_{t-s}f+2\Gamma(P_{t-s}f).\]
Note that $\Delta P_tf\in \ell^2_m$ since $P_tf\in D(\Delta)$. Besides, by Lemma \ref{lem:0}, and using that $\ell^1_m\subset \ell^2_m$ for non-degenerate measure, we have $\Gamma(P_{t-s}f)\in \ell^2_m$. Combining these facts, we get what we desire.

By (2), we have
\[\begin{split}
\frac{d}{ds}\phi=2P_s(\Gamma(P_{t-s}f))
&\geq 2e^{2Ks}\Gamma(P_sP_{t-s}f)+\frac{2 (e^{2Ks}-1)}{Kn}(\Delta P_sP_{t-s}f)^2\\
&=2e^{2Ks}\Gamma(P_tf)+\frac{2(e^{2Ks}-1)}{Kn}(\Delta P_tf)^2,
\end{split}\]
integration over $s$ on $[0, t]$ yields the result (3).
\end{proof}
A first interesting consequence of the above functional inequalities is the following Pseudo-Poincar\'e inequality.
\begin{theorem}\label{thm:poincare}
If the infinite graph $G=(V,E,m,\omega)$ satisfies $CD(\infty,K)$ for some $K\in \mathbf{R}$ and the assumption $(A)$, then for any $f\in C_0(V)$ and any $0\leq t\leq \frac{1}{|2K|}$, we have
\[\|f-P_t f\|_{\ell^1_m}\leq 4\sqrt{t}\|\sqrt{\Gamma(f)}\|_{\ell^1_m}.\]
\end{theorem}
Note that the restriction on $t$ applies only when $K$ is negative.

\begin{proof}
Let $g=sgn(f-P_t f)+1$. From Corollary 2.9 in \cite{KL12}, $P_t$ is positivity improving, i.e., $P_t$ map nonnegative nontrivial $\ell^2_m$-functions to strictly positive functions, so $g\in C_0(V)$.
We first claim that for any $f\in C_0(V)$,
\begin{equation}\label{eq:inv}
\sum_{x\in V}m(x)P_tf(x)=\sum_{x\in V}m(x)f(x).
\end{equation}
In fact, let $0\leq \chi_k\in C_0(V)$ be a non-decreasing sequence satisfying $\chi_k\rightarrow \mathbf{1}$ when $k$ goes to infinity. Since $P_t f\in \ell^1_m$,
then the dominated convergence theorem
yields that
\[\sum_{x\in V}m(x)P_t f(x)\chi_k(x)\rightarrow \sum_{x\in V}m(x)P_t f(x),\]
and also $P_t\chi_k\rightarrow P_t\mathbf{1}$
pointwise. From \cite{W09}, the heat semigroup associated to bounded Laplacian is stochastically complete for all graphs, i.e., $P_t\mathbf{1}=\mathbf{1}$.
Moreover, $G$ is stochastically complete under the assumption $(A)$, see \cite{HL17} for unbounded Laplacians under assumption $(A_2)$ and $K\geq0$.
Therefore, we can compute that
\[\begin{split}
\sum_{x\in V}m(x)P_tf(x)
&=\lim_{k\rightarrow \infty}\sum_{x\in V}m(x)P_tf(x)\chi_k(x)\\
&=\lim_{k\rightarrow \infty}\sum_{x\in V}m(x)f(x)P_t\chi_k(x)\\
&=\sum_{x\in V}m(x)f(x)\lim_{k\rightarrow \infty}P_t\chi_k(x)\\
&=\sum_{x\in V}m(x)f(x),
\end{split}\]
where we use that $P_t$ is self-adjoint on functions in $\ell^2_m$ in the second step.
Then
\[\|f-P_t f\|_{\ell^1_m}=\sum_{x\in V}m(x)g(x)(f-P_t f)(x).\]

Since $g\in C_0(V)$, we get $P_t g\in \ell^2_m$ and $\Gamma (P_t g)\in\ell^1_m$. Moreover, from Theorem 6 in \cite{KL12}, we know that $P_t g$ can be approximated in $Q$-norm by functions with finite support. By definition, we can conclude that $P_t g\in D(Q)$. Therefore, from Green's formula, we obtain
\begin{equation}\label{eq:3}
\begin{split}
\sum_{x\in V}m(x)g(x)(f-P_t f)(x)
&=-\sum_{x\in V}m(x)g(x)\int_0^t\frac{d}{ds} P_s f(x)ds\\
&=-\int_0^t\sum_{x\in V}m(x)g(x)P_s (\Delta f)(x)ds\\
&=-\int_0^t\sum_{x\in V}m(x) \Delta f(x)P_sg(x)ds\\
&=\int_0^t\sum_{x\in V}m(x) \Gamma(f,P_sg)(x)ds\\
&\leq\int_0^t\sum_{x\in V}m(x)\sqrt{\Gamma(f)(x)\Gamma(P_sg)(x)}ds.
\end{split}
\end{equation}
Hence,
\begin{equation}\label{eq:4}
\|f-P_t f\|_{\ell^1_m}\leq\|\Gamma(f)\|_{\ell^1_m}\int_0^t\|\sqrt{\Gamma(P_sg)}\|_\infty ds.
\end{equation}
Since $\int_0^t2e^{2Ks}ds=(e^{2Kt}-1)/K$ when $K\neq0$, and equal to $2t$ when $K=0$,
we can conclude that $\int_0^t2e^{2Ks}ds\geq t$ for $0\leq t\leq \frac{1}{2|K|}$ with any $K\in \mathbf{R}$. Therefore, applying (3) in Proposition \ref{pro:eqi} and ignoring the nonnegative item $(P_t g)^2$, 
\[\|\sqrt{\Gamma(P_sg)}\|_\infty\leq\frac{1}{\sqrt{s}}\|\sqrt{P_s(g^2)}\|_\infty\leq\frac{1}{\sqrt{s}}\|g\|_\infty=\frac{2}{\sqrt{s}},\]
then yields what we desire.
\end{proof}
\begin{remark}\label{re:1}
In fact, from \eqref{eq:4} and applying (3) in Proposition \ref{pro:eqi}, by direct computations, we obtain
\[
\|f-P_t f\|_{\ell^1_m}\leq\frac{2}{\sqrt{K}}\left(\pi-2\arcsin e^{-Kt}\right)\|\Gamma(f)\|_{\ell^1_m},
\]
which will be useful later.
\end{remark}


\begin{proof} [Proof of Theorem \ref{th:buser}]
For any non-empty finite subset $U\subset V$, let $f$ be the function $\mathbf{1}_U$, which equals to $1$ when $x\in U$ and $0$ otherwise, applying the Theorem \ref{thm:poincare}, we get
\begin{equation}\label{eq:upper}
\|\sqrt{\Gamma(\mathbf{1}_U)}\|_{\ell^1_m}
=\sum_{x\in V}\sqrt{\frac{m(x)}{2}}\sqrt{\sum_{y\sim x}\omega_{xy}(\mathbf{1}_U(y)-\mathbf{1}_U(x))^2}
\leq\sqrt{2D_\omega}|\partial U|.
\end{equation}
Therefore we have
\begin{equation}\label{eq:b1}
\|\mathbf{1}_U-P_t (\mathbf{1}_U)\|_{\ell^1_m}\leq 4\sqrt{2D_\omega t}|\partial U|.
\end{equation}

We further claim that
\begin{equation}\label{eq:b2}
\|\mathbf{1}_U-P_t
(\mathbf{1}_U)\|_{\ell^1_m}\geq2(1-e^{-\lambda t})|U|.
\end{equation}
From \eqref{eq:inv}, we have
\[|U|=\sum_{x\in V}m(x)P_t (\mathbf{1}_U)(x)=\sum_{x\in U}m(x)P_t (\mathbf{1}_U)(x)+\sum_{x\in V/U}m(x)P_t (\mathbf{1}_U)(x).\]
Then
\[\begin{split}
\|\mathbf{1}_U-P_t (\mathbf{1}_U)\|_{\ell^1_m}
&=|U|-\sum_{x\in U}m(x)P_t (\mathbf{1}_U)(x)+\sum_{x\in V/U}m(x)P_t (\mathbf{1}_U)(x)\\
&=2\left(|U|-\sum_{x\in V}m(x)\mathbf{1}_U(x)P_t (\mathbf{1}_U)(x)\right).
\end{split}\]
Moreover, let $f=g=\mathbf{1}_U$ and $A=B=U$ in \eqref{eq:heatkernel} in Lemma \ref{lem:heatkernel}, note that $\rho(U,U)=0$, which leads to
\[\sum_{x\in V}m(x)\mathbf{1}_U(x)P_t (\mathbf{1}_U)(x)
\leq e^{-\lambda t}|U|.\]
By simple computations, we prove the previous claim.

Combining inequalities \eqref{eq:b1} and \eqref{eq:b2}, we obtain
\[\frac{|\partial U|}{|U|}\geq\frac{1-e^{-\lambda t}}{C\sqrt{t}},\]
where $C=4\sqrt{2D_\omega}$.
So one has
\[h(V) \geq\sup_{0\leq t\leq \frac{1}{2|K|}}\frac{1-e^{-\lambda t}}{C\sqrt{t}}.\]
If $\lambda\geq 2|K|$, we choose $t=1/\lambda$ to yield the result. If $\lambda\leq 2|K|$, we select $t=1/(2|K|)$, then $1-e^{-\lambda t}\geq\lambda/(4|K|),$ and finally deduce the theorem.
\end{proof}

From here, we consider the graph with positive curvature, i.e. $K > 0$. Another interesting consequence of the above equivalent semigroup properties in Proposition \ref{pro:eqi} is the fact that the measure on the graph is finite.
\begin{proof} [Proof of Theorem \ref{thm:finite}]



As we mentioned before, we need only to consider unbounded Laplacians with assumption $(A_2)$.

We first ignore the nonnegative item $\Gamma(P_t f)$ in (2) of Proposition \ref{pro:eqi}, and obtain
\begin{equation}\label{eq:eqipo}
\frac{1-e^{-2Kt}}{Kn}(\Delta P_tf)^2\leq e^{-2Kt}P_t(\Gamma(f)).
\end{equation}
Taking the square root of \eqref{eq:eqipo}, and using $\|P_t(\Gamma(f))\|_\infty\leq\|\Gamma(f)\|_\infty$, yields
\[\left|\frac{d}{dt} P_tf\right|=|\Delta P_tf|\leq \sqrt{\frac{n K}{\delta}\|\Gamma(f)\|_\infty}\sqrt{\frac{1}{e^{2Kt}-1}}\rightarrow 0,~~t\rightarrow+\infty.\]
Therefore, fix $f\in C_0(V)$, the limit of $P_tf(x)$ exists and is finite when $t\rightarrow+\infty$.

Next, we prove that this limit is not related to $x$. From the same semigroup property of $CD(n,K)$, by dropping $-\frac{1-e^{-2Kt}}{Kn}(\Delta P_tf)^2$ in (2) of Proposition \ref{pro:eqi}, we have
\[\Gamma(P_tf)\leq e^{-2Kt}P_t(\Gamma(f)),\]
which implies $\|\Gamma(P_tf)\|_\infty\rightarrow 0$ when $t\rightarrow +\infty$, and since the graph $G$ is connected, we get
\[|P_tf(x)-P_tf(y)|\rightarrow0,~~\mbox{ for any}~~ x,y \in V.\]
Then we can conclude that $P_tf$ converges to a constant when $t\rightarrow \infty$, and it is easy to see this convergence is uniform.

Let us now assume that $|V|=+\infty$. This implies in particular that $\lim_{t\rightarrow+\infty}P_tf=0$ because no
constant except $0$ is in $\ell^p_m, p\in[1,\infty)$.
Applying (2) of Proposition \ref{pro:eqi} and \eqref{eq:3}, for any $g\in C_0(V)$, we have
\[\left|\sum_{x\in V}m(x)g(x)(f-P_t f)(x)\right|\leq\|\sqrt{\Gamma(f)}\|_\infty\int_0^t\|\sqrt{\Gamma(P_sg)}\|_{\ell^1_m}ds
\leq\|\sqrt{\Gamma(f)}\|_\infty\int_0^te^{-Ks}ds\|\sqrt{\Gamma(g)}\|_{\ell^1_m}.\]
Letting $t \rightarrow+\infty$  yields
\[\left|\sum_{x\in V}m(x)g(x)f(x)\right|\leq\frac{1}{K}\|(\Gamma(g))\|_{\ell^1_m}\|\sqrt{\Gamma(f)}\|_\infty.\]
Let us assume $g\geq 0, g\not \equiv0$ and take for $f$ the sequence $\eta_k$ from Assumption $(A_2)$. Letting $k\rightarrow\infty$, we deduce
\[\sum_{x\in V}m(x)g(x)\leq0,\]
which is impossible. Therefore, $|V|<+\infty.$
Moreover, from \eqref{eq:inv} and letting $t\rightarrow+\infty$, \eqref{eq:infinity} is proved.
\end{proof}
The above result allows us to deduce the following Poincar\'e inequality. From it, we can estimate Cheeger constant in terms of curvature.
\begin{proof}[Proof of Theorem \ref{thm:chegger}]
From \eqref{eq:infinity} and Remark \ref{re:1}, by letting $t\rightarrow+\infty$, it is clear that
\begin{equation*}\label{eq:poincare}
\|f-f_V\|_{\ell^1_m}\leq\frac{2\pi}{\sqrt{K}}\|\sqrt{\Gamma(f)}\|_{\ell^1_m},
\end{equation*}
where $f_V=\frac{1}{|V|}\sum_{x\in V}m(x)f(x)$.

For any non-empty subset $U\subset V$, let $f=\mathbf{1}_U$ in the above inequality, we have
\[\left\|\mathbf{1}_U-\frac{|U|}{|V|}\right\|_{\ell^1_m}=\frac{|U|\cdot|V/U|}{|V|}\geq \min\{|U|,|V/U|\}.\]
Combining with \eqref{eq:upper}, yields
\[\frac{|\partial U|}{\min\{|U|,|V/U|\}}\geq\frac{\sqrt{K}}{2\pi \sqrt{2D_\omega}}.\]
By definition, the theorem is proved.
\end{proof}

\bibliographystyle{amsalpha}

Shuang Liu,\\
Yau Mathematical Sciences Center, Tsinghua University, Beijing, China\\
\textsf{shuangliu@mail.tsinghua.edu.cn}\\
\end{document}